\documentclass[12pt]{amsart}
\usepackage{graphicx} 

\usepackage[utf8]{inputenc}

\usepackage{amsthm}
\usepackage{amsmath}
\usepackage{color}

\usepackage{tikz}									
\usetikzlibrary{matrix}
\usetikzlibrary{patterns}
\usetikzlibrary{matrix}
\usetikzlibrary{positioning}
\usetikzlibrary{decorations.pathmorphing}
\usetikzlibrary{cd}

\usepackage[all]{xy}
\usepackage[margin=1in]{geometry}
\usepackage{verbatim}
\usepackage{mathrsfs}
\usepackage{enumitem}
\usepackage[colorlinks=true,allcolors=purple]{hyperref}
\usepackage{amssymb}
\usepackage{mathtools}

\usetikzlibrary{decorations.markings}
\usetikzlibrary{decorations.pathreplacing}
\usetikzlibrary{positioning}
\usetikzlibrary{arrows,calc,patterns}
\usetikzlibrary{intersections,through}

\usepackage{pgfplots}\pgfplotsset{compat=1.13}

\theoremstyle{plain}

\newtheorem{theorem}{Theorem}
\newtheorem*{theorem*}{Theorem}

\newtheorem{proposition}[theorem]{Proposition}
\newtheorem{corollary}[theorem]{Corollary}
\newtheorem{lemma}[theorem]{Lemma}

\newtheorem*{maintheorem}{Theorem}

\theoremstyle{definition}

\newtheorem{example}[theorem]{Example}

\theoremstyle{remark}

\newtheorem{remark}[theorem]{Remark}

\numberwithin{theorem}{section}
\numberwithin{equation}{section}

\def\O{\mathcal{O}}

\def\T{\mathbb{T}}

\def\B{\mathbb{B}}

\def\Z{\mathbb{Z}}
\def\Bases{\mathcal{B}}

\def\max{{\rm max}}

\def\bv{{\bf v}}
\def\bw{{\bf w}}
\def\be{{\bf e}}
\def\bfn{{\bf f}}
\def\BG{\mathbb{B}[G]}
\def\BZn{\mathbb{B}[\Z_n]}
\def\BZp{\mathbb{B}[\Z_p]}

\title[Trop. subrepresentations of the boolean regular rep. in low dim. ]{Tropical subrepresentations of the boolean regular representation in low dimension}

\author{Steffen Marcus}
\address[Marcus]{The College of New Jersey\\
Ewing, NJ 08628\\ USA}
\email{marcuss@tcnj.edu}

\author{Cameron Phillips}
\address[Phillips]{The College of New Jersey\\
Ewing, NJ 08628\\ USA}
\email{phillic2@tcnj.edu}

\date{\today}

\begin{document}

\maketitle

\section{Introduction}
In \cite{GM2020}, Giansiracusa--Manaker develop the theory of group representations over a fixed idempotent semifield $\mathbb{S}$, i.e.\ \emph{tropical} representation theory. This continues efforts in recent years to pursue ideas in scheme theoretic tropicalization, including \cite{Frenk2013, GG2016, GG2017, MR2018, JUN2018, CGM2020, L2023, BB2019,fink2024} among many others. In this theory a linear representation $G\longrightarrow {\rm GL}(\mathbb{S}^n)$ of a group $G$ over $\mathbb{S}$ is described by monomial matrices, and a \emph{tropical subrepresentation} $T\subseteq \mathbb{S}^n$ is a $G$-invariant tropical linear space in $\mathbb{S}^n$. The constant-coefficient case features tropical representations over the boolean semifield $\B=\{-\infty,0\}$ consisting of the tropical additive and multiplicative identities. This setting highlights a unique interaction between group theory and matroid theory, since tropical representations $T\subseteq \mathbb{S}^n$ correspond to group homomorphisms $G\longrightarrow {\rm Aut}(M_T)$ to the automorphism group of the corresponding matroid. 

Given a finite group $G$ the \emph{boolean regular representation} $\BG$ is the tropicalization of the regular representation $\mathbb{C}[G]$ and can be described per usual by considering the $\B$ linear span
\[\BG = \left\{ \sum_{g\in G} c_g\be_g \right\}\]
of basis elements indexed by $G$ under the action determined by left-multiplication in $G$. As a central example \cite[Theorem~B]{GM2020} the authors begin the study of tropical subrepresentations of $\BG$ and their \emph{realizability}, that is, whether they arise as the tropicalization of a classical subrepresentation of $\mathbb{C}[G]$. In particular when $G\cong \Z_p$ is cyclic of prime order they prove that there is only one realizable subrepresentation in each dimension $1\leq d \leq p$ corresponding to the uniform matroid $U_{d,p}$. In dimension 2 they show that that there are no other non-realizable subrepresentations of $\BZp$ for a specified infinite collection of primes. They conjecture \cite[Conjecture~4.1.6]{GM2020} that there is only one two-dimensional subrepresentation of $\BZn$ if and only if n is prime.

The goal of this paper is to continue the classification of tropical subrepresentations of $\BG$. Our main result in two-dimensions classifies all two-dimensional tropical subrepresentations for any finite group $G$ and provides a direct correspondence with the proper subgroups of $G$. 
\begin{maintheorem}[Theorem~\ref{thm:main}] Let $G$ be a finite group. Two-dimensional tropical subrepresentations of $\BG$ correspond bijectively to proper subgroups $H\subset G$. The set of bases of the corresponding matroids are explicitly presented as a union of $G$-orbits 
    \[\bigcup_{g\in G-H}\{\{a,ag\}|a\in G\}\]
    for the induced $G$-action on subsets of $G$ size 2.
\end{maintheorem}
This theorem is proven in Section~\ref{sec:2d}. It completely solves the problem in dimension 2 and confirms conjecture \cite[Conjecture~4.1.6]{GM2020} as the special case where $G\cong\Z_p$. The union of $G$-orbits explicitly describes the set of bases for each matroid on the ground set $G$ corresponding to a tropical subrepresentation. Our proof generalizes the methods in \cite[Section~4]{GM2020} to arbitrary groups, is surprisingly elementary, and demonstrates a direct interplay between the basis exchange axiom and the group structure of $G$. In higher dimensions, this interplay and the combinatorics involved present a more significant challenge. 

To this end, in Section~\ref{sec:3d} we make progress on the dimension 3 case. Our first result in dimension 3 shows that for any subgroup $H\subseteq G$ of index larger than 2, there exists a tropical subrepresentation of $\BG$ with bases of its corresponding matroid identified in a similar way to Theorem~\ref{thm:main}.

\begin{maintheorem}[Theorem~\ref{thm:dim3subgroups}] 
Let $G$ be a finite group, and let $H$ be a subgroup of $G$ with $[G:H]>2$. There exists a three-dimensional tropical subrepresentations of $\BG$ for which the set of bases of the corresponding matroid is explicitly presented as a union of $G$-orbits 
    \[\bigcup_{g,h,g^{-1}h\in G-H} \{\{a,ag,ah\}|a\in G\}\]
    for the induced $G$-action on subsets of $G$ size 3.
\end{maintheorem}

In contrast to dimension 2, this is not an equivalence -- the combinatorics in higher dimension seem to allow for a wider collection of tropical subrepresentations. In Section~\ref{sec:3dcyclic} we investigate this further in the cyclic case. Our first result for cyclic groups identifies a common subset of the set of bases for every matroid corresponding to a three-dimensional subrepresentation of $\B[\Z_n]$. Denote by $\Z_n^\times$ the set of units in $\Z_n$.

\begin{maintheorem}[Theorem~\ref{thm:3d}] Let $G=\Z_n$ be a finite cyclic group. The set of bases $\Bases$ of any matroid corresponding to a three-dimensional subrepresentation contains the union 
    \[\bigcup_{u\in \Z_n^\times} \{ \{a,u+a,2u+a\} | a \in \Z_n \} \]
    of orbits for the induced $\Z_n$-action on subsets of $\Z_n$ of size 3.
\end{maintheorem}

Our second result for cyclic groups provides the set of bases for a number of tropical subrepresentations of $\BZn$ not corresponding to subgroups. Denote by ${[n]}\choose{3}$ the set of subsets of $[n]=\{1,\ldots,n\}$ of size 3. In Theorem~\ref{thm:3d2} these sets of bases take the form, for any unit $u\in\Z_n^\times$,
\[{{[n]}\choose{3}}-\left\{\{a,u+a,ku+a\}|a\in\Z_n\right\}\]
for various values of $k$ determined by the combinatorics of the problem. Here we are simply excluding one specific orbit of the induced $\Z_n$-action.
For $n>5$ these are the bases of non-uniform matroids. This confirms that $\B[\Z_n]$ contains non-uniform three-dimensional tropical subrepresentations for all $n>5$, even for $n$ a prime, which directly contrasts with the two-dimensional case. 

In the boolean setting, tropical representation theory can be recast as a question about group actions on matroids, and this is our approach. Our attention is restricted to the classification question as it provides a facinating playground in which group theory and matroid theory seem to interact deeply with interesting combinatorics. Realizability of the subrepresentations constructed here, a continuation of the classification in dimensions 3 and higher, and a comparison with the classical case are all ripe questions that we leave for future work.

\subsection*{Acknowledgements} We are grateful to Andrew Clifford, Noah Giansiracusa, and Thomas Hagedorn for interesting and helpful discussions benefiting this project.

\section{Background}

In this section we briefly summarize the necessary theory from \cite{GM2020}. For more detailed background, we refer the reader to this paper, as well as \cite{CGM2020, GG2017}. We assume a familiarity with matroid theory, see \cite{Oxley2011} for further reading.

\subsection{The tropical booleans}

In \cite{GM2020} the finite-dimensional representation theory of groups is developed over an arbitrary idempotent semifield $(\mathbb{S},+,\cdot,0,1)$, meaning additive inverses may not exist and $s+s=s$ for all $s\in\mathbb{S}$. This is an appropriate generalized setting for tropical algebra as the \textit{tropical numbers} \[\T=(\mathbb{R}\cup\{-\infty\},\oplus,\odot, -\infty, 0)\] form an idempotent semifield consisting of the set $\mathbb{R}\cup\{-\infty\}$ under the operations $a\oplus b=\max(a,b)$ and $a\odot b=a+b$. Note that $x\oplus y=\max(x,y)\ge x$ for all $x,y\in \T$ and, aside from $-\infty$, no element has an additive inverse. 

The boolean semifield \[\B=\{0,1\}\subseteq\mathbb{S}\] consists of the additive and multiplicative identity elements alone and forms the initial object of the category. We will restrict our attention to the booleans throughout.

\subsection{Exterior powers}

Let $\B^n$ be a free $\B$ module with basis elements  $\be_1,\ldots,\be_n$. In \cite{GG2017}, an exterior algebra formalism for semirings introduces free modules $\bigwedge^d\B^n$ for $d=1,\ldots,n$ given by a basis of the form $\be_{i_1}\wedge\cdots\wedge \be_{i_d}$ where each $\be_{i_k}$ are basis elements of $\B^n$. Recall the two standard properties or wedge product: $\bv\wedge \bv=0$ and $\bv\wedge \bw=-\bw\wedge \bv$.  
In this context the second property becomes $\bv\wedge \bw=\bw\wedge\bv$.

For any positive integer $n$, set $[n]=\{1,\ldots, n\}$. For any finite set $A$ denote by ${{A}\choose{d}}$ the set of subsets of $A$ of cardinality $d$. The basis of $\bigwedge^d\B^n$ corresponds to the set ${{[n]}\choose{d}}$ by taking indices, and vectors in $\bigwedge^d\B^n$ simply indicate a subset of the basis vectors of $\bigwedge^d\B^n$. This gives a correspondence betwen vectors $v\in \bigwedge^d\B^n$ and subsets of ${[n]}\choose{d}$. More generally given any finite group $G$ as an index set, the same correspondence can be formed between vectors $v\in \bigwedge^d\B^{|G|}$ and subsets of ${G}\choose{d}$. Moreover, the effect of projectivization on $\bigwedge^d\B^n$ disappears since $\B^*=\B-0=\{1\}$, so $\mathbb{P}(\bigwedge^d\B^n)\cong \bigwedge^d\B^n$.

\subsection{Tropical Linear Spaces and Matroids}\label{sec:boolwedge}

A nonzero vector 
\[\bv=\sum_{I\in{{[n]}\choose{d}}} v_I\be_I\in \bigwedge^d\B^n\]
is a \emph{tropical Pl\"{u}cker vector over $\mathbb{B}$} of rank $d$ if it satisfies the \emph{Tropical Pl\"ucker Relations} over $\mathbb{B}$, that is, for all $A\in{{[n]}\choose{d+1}},\ B\in{{[n]}\choose{d-1}},\ j\in A-B$,
\[
\sum_{i\in A-B}v_{A-\{i\}}v_{B\cup\{i\}}=\sum_{i\in A-B-j}v_{A-\{i\}}v_{B\cup\{i\}}.
\]
Formally, a vector will satisfy the tropical Pl\"ucker relations if whenever $v_{A-i}v_{B+i}=1$ for some $i$, then for some $j\ne i$ we must have $v_{A-j}v_{B+j}=1$. Indeed, tropical Pl\"{u}cker vectors in $\bigwedge^d\B^n$ are precisely the vectors which correspond to choices of subsets in ${[n]}\choose{d}$ that satisfy the strong basis exchange axiom, which we recall here: for all $X,Y\in\Bases,\ i\in X-Y,$ there is some $j\in Y-X$ with $X-i+j,\ Y-j+i\in\Bases$. Thus $\bv$ is a tropical Pl\"{u}cker vector if and only if it is the basis indicator vector of a matroid over the ground set $[n]$, i.e. $\Bases=\{I\mid V_I=1\}$ is the basis set of a matroid.

The set of all tropical Pl\"{u}cker vectors are in a bijective correspondence with submodules $L\subseteq\B^n$. These submodules are the \emph{tropical linear spaces} in $\B^n$. In the boolean setting this correspondence takes a satisfying form. A tropical Pl\"{u}cker vector $\bv\in \bigwedge^d\B^n$ corresponds to a matroid $M_\bv=([n],\Bases_v)$ with bases $\Bases_v$ indicated by the vector $\bv$, and this data corresponds to a tropical linear subspace $L_\bv \subseteq \B^n$ given by the $\B$-linear span of the vectors $\sum_i w_i\be_i$ indicating the \emph{cocircuits} of $M_\bv$. Thus 
\[\bv\in \bigwedge^d\B^n \iff M_\bv=([n],\Bases_v) \iff L_\bv \subseteq \B^n.\]
See \cite{GG2017} for a more thorough discussion.

\subsection{Linear representations and $G$-fixed Pl\"{u}cker vectors}
Let $G$ be an arbitrary group and let $V\cong \B^n$ be a free module of rank $n$ and let $[n]$ be the index set for the chosen basis of $V$. A (linear) \emph{tropical representation of $G$ over $\B$} is a group homomorphism $\rho:G\longrightarrow {\rm GL}(V)$. Note that ${\rm GL}(V)$ consists solely of permutation matrices, so over the booleans the theory boils down to permutation representations. A \emph{tropical subrepresentation} of $\rho$ is given by restricting the $G-$action to a $G-$invariant tropical linear subspace $L\subseteq V$. The dimension of a tropical subrepresentation is the rank of the tropical linear subspace $L$. This also equals the rank of the corresponding matroid under the correspondence described in Section~\ref{sec:boolwedge}.

For each $1\leq d \leq n$, a given tropical representation $\rho:G\longrightarrow {\rm GL}(V)$ induces a representation  $G\longrightarrow {\rm GL}(\bigwedge^d V)$ which restricts to a $G$-action on the \emph{Dressian} ${\rm Dr}(d,n)\subseteq\bigwedge^d V$, that is, the subset consisting of tropical Pl\"{u}cker vectors of rank $d$.  
We also have a compatability of $G$-actions: whevever $L_v\subseteq V$ is a tropical linear space in $V$ associated to the tropical Pl\"{u}cker vector $\bv\in \bigwedge^d V$ then for all $g\in G$ we have $L_{g\cdot \bv} = g\cdot L_\bv$.  
Thus tropical subrepresentations of dimension $d$ correspond to $G$ fixed tropical Pl\"{u}cker vectors $\bv\in \bigwedge^d V$. This crucial observation from \cite{CGM2020} is shown for all idempotent semifields and is central to our approach. We highlight it here in the Boolean case.

\begin{theorem}[\cite{GM2020} Theorem~A(1)] For any $1\leq d \leq n$, the induced linear representation on $\bigwedge^d \B^n$ restricts to a linear action on the Dressian ${\rm Dr}(d,n)\subseteq \bigwedge^d \B^n$, and $d$-dimensional tropical subrepresentations in $\B^n$ are equivalent to $G$-fixed points in ${\rm Dr}(d,n)$.    
\end{theorem}

Thus the problem of classifying d-dimensional tropical subrepresentation of $\rho$ is equivalent to identifying all rank $d$ tropical Pl\"{u}cker vectors fixed by the induced $G$-action. Moreover, this reduces further to a problem in matroid theory. We recall that  an automorphism of a matroid is a permutation of the ground set sending independent sets to independent sets, and thus bases to bases and cocircuits to cocircuts. Tropical linear subspaces $L_\bv\subseteq V$ correspond to matroids $M_\bv$ over the ground set $[n]$. A tropical representation $\rho:G\longrightarrow{\rm GL}(V)$ induces a $G$-action on $[n]$, and $\rho$ factors through the automorphism group of $M_\bv$ if and only if $L_\bv$ is $G$-invariant. Since a tropical Pl\"{u}cker vector is the basis indicator vector for $M_\bv$, a $G$-fixed tropical Pl\"{u}cker vector $\bv$ corresponds to a set of bases $\Bases_\bv$ invariant under the $G$-action on $[n]$. Thus to classify d-dimensional tropical subrepresentation of $\rho$, it is enough to classify matroids rank $d$ matroids $([n],\Bases)$ for which $\Bases\subseteq {{[n]}\choose{d}}$ is invariant under the induced $G$-action.

\subsection{The regular representation}

Let $G$ be a finite group. For each $g\in G$ we write $\be_g$ for a basis element represented by the group element $g$ and denote by \[\BG = \left\{ \sum_{g\in G} c_g\be_g \right\}\] the $\B$-linear span of these basis elements. The regular representation is given by taking the $G$ action by left multiplication on the basis elements of $\BG$:
\[g\cdot \be_a=\be_{g a},\ g,a\in G.\]
Fix $1\leq d \leq n$. We can extend this action to the basis of $\bigwedge^d\BG$ in the following manner:

\[g\cdot \left(\be_{a_1}\wedge \cdots \wedge \be_{a_d}\right)=\be_{ga_1}\wedge\cdots\wedge \be_{ga_d},\ g\in G, \{a_1,\ldots,a_d\}\in {{G}\choose{d}}.\]
This extends linearly to an action on $\bigwedge^d\BG$. 

In parallel, consider the $G$ action on itself by left multiplication:
\[g\cdot a=ga,\ g,a\in G.\]
This extends to ${G}\choose{d}$ in the following manner:
\[g\cdot\{a_1,\ldots,a_d\}=\{ga_1,\ldots,ga_d\},\ g\in G,\ \{a_1,\ldots,a_d\}\in {{G}\choose{d}}.\]
There is an obvious bijection from the basis of $\bigwedge^d\BG$ to the set of subsets of $G$ of size d along which these two actions are equivalent. The bijection extends to one between $\bigwedge^d\BG$ and the power set of ${G}\choose{d}$, by sending the d-vector \[\sum_{I\in{{G}\choose{d}}} v_I\be_I\in\bigwedge^d\BG\] to the set of subsets $\left\{I|v_I=1\right\}$ it indicates. Hence terms in the tropical Pl\"ucker vector $\bv=\sum_{I\in{{G}\choose{d}}} v_I\be_I$ associated with the matroid $M_\bv=(G,\Bases_\bv)$ will have \[v_I=1\iff I\in\Bases\] and thus can be written as $\bv=\sum_{I\in\Bases_\bv}\be_I$.

\section{2-dimensional Subrepresentations of $\BG$} \label{sec:2d}
In this section we classify all 2-dimensional tropical subrepresentations of $\BG$ for an abritrary finite group $G$ and provide some examples. By \cite[Theorem~A(1)]{GM2020} these correspond to rank 2 tropical Pl\"ucker vectors fixed by the induced action of $G$ on $\bigwedge^2\BG$. They are studied by taking unions of orbits of the equivalent action on ${G}\choose{2}$. This amounts to searching for unions of orbits of the $G$ action on ${G}\choose{2}$ satisfying the strong basis exchange axiom.

\subsection{Orbits and their properties} 
To identify the orbits of the $G$ action on ${{G}\choose{2}}$ note that for any subset $\{a,b\}\in{{G}\choose{2}}$, the product $a^{-1}b$ is preserved under the $G$ action on ${{G}\choose{2}}$ since \[(ga)^{-1}gb=a^{-1}g^{-1}gb=a^{-1}b.\] We think of $a^{-1}b$ as a generalized ``difference" (see the example of $\Z_n$ in Section~\ref{ex:Zn} below). Moreover, we have that every pair with the same generalized difference is in the same orbit. To see this, let $\{a_1,b_1\},\{a_2,b_2\}\in {{G}\choose{2}}$ with \[a_1^{-1}b_1=a_2^{-1}b_2,\] giving 
\[a_2a_1^{-1}b_1=b_2.\] Then $a_2a_1^{-1}$ acting on $\{a_1,b_1\}$ gives
\[a_2a_1^{-1}\cdot\{a_1,b_1\}=\{a_2a_1^{-1}a_1,a_2a_1^{-1}b_1\}=\{a_2,b_2\}.\]

Thus the orbits are given by the sets
 \[f_g=\{\{a,ag\}|a \in G\} = \{\{a,b\}|a,b \in G, a^{-1}b = g\}\]
consisting of pairs $\{a,b\}$ with a fixed difference $g=a^{-1}b$ between the two elements. Note that $g=e$ gives $f_e=\emptyset$ and each orbit $f_g$ determines a $G$ fixed vector $\bfn_g\in \bigwedge^2\BG$. Denote by $\O_G$ the set of orbits of the $G$ action on ${{G}\choose{2}}$. Each orbit determines a $G$ fixed vector in $\bigwedge^2\BG$, denoted
 \[\bfn_g=\sum_{\substack{a\in G}} \be_a\wedge \be_{ag}.\]
We now highlight salient properties of the orbits $f_g$. 

\begin{proposition}\label{prop:orbit-1}
    For all $g\in G$, $f_g=f_{g^{-1}}$.
\end{proposition}
\begin{proof}
    To show $f_g\subseteq f_{g^{-1}}$, let $\{a,b\}$ be an arbitrary set in $f_g$. Then $g=a^{-1}b$ by definition and $\{a,b\}=\{b,a\}\in f_{b^{-1}a}=f_{(a^{-1}b)^{-1}}=f_{g^{-1}}$. The containment in the other direction is similar.
\end{proof}

For reference, the matroid basis axioms are: 
\begin{enumerate}
    \item $\Bases\ne\emptyset$
    \item if $A,B\in\Bases$ then $|A|=|B|$
    \item if $A,B\in\Bases$ and $x\in A-B$, then there is some $y\in B-A$ with $A-x+y\in\Bases$ (strong basis exchange).
\end{enumerate}

Let $A=\{a_1,a_2\}$ and $B=\{b_1,b_2\}$ be bases. There are three cases:
\begin{enumerate}
    \item[Case 1:] $A=B$. In this case there is no $x$ in $A-B$ and basis exchange succeeds.
    \item[Case 2:] $A,B$ are unequal but not disjoint, so they share exactly one element since they both have size 2. Without loss of generality, let $a_1=b_1$. Then $x=a_2$, $y=b_2$, and this forces $A-a_2+b_2$ to be in $\Bases$. However, this set is $\{a_1,b_2\}$, but since $a_1=b_1$ the set is just $\{b_1,b_2\}=B$.
    \item[Case 3:] $A,B$ are disjoint. Without loss of generality,  pick $x=a_2$; since the sets are disjoint, $y$ could be either $b_1$ or $b_2$. Then we have $\{a_1,b_1\}$ or $\{a_1,b_2\}$ is in $\Bases$.
\end{enumerate}
Since only the third case results in including sets not already in $\Bases$, we will only consider basis exchange between disjoint subsets of size 2 when dealing with dimension 2. The following proposition shows how basis exchange can be applied to the orbits of the $G$ action on ${{G}\choose{2}}$.

\begin{proposition}\label{prop:orbitgh} Let $\Bases$ be the set of bases of a matroid $M_\bv=(G,\Bases)$ corresponding to a $G$-fixed tropical Pl\"ucker vector $\bv\in \bigwedge^2\BG$. Then
    \[f_{gh}\subseteq\Bases\implies f_g\subseteq\Bases\text{ or }f_h\subseteq\Bases.\]
\end{proposition}
\begin{proof}We note that $\{e,gh\},\{g^{-1},h\}\in f_{gh}$. These sets are disjoint as long as $g^{-1},h\neq e$, $gh\neq g^{-1}$ and $gh\neq h$. In any of these cases, $f_{gh}=f_g$ or $f_h$ and the proposition is true.
Otherwise, the sets are disjoint and we can use basis exchange to deduce \[\{e,g^{-1}\}\in \Bases\text{ or }\{e,h\}\in\Bases.\] But $\{e,g^{-1}\}\in f_{g^{-1}}=f_g$ by Proposition~\ref{prop:orbit-1} and $\{e,h\}\in f_h$. Therefore we have $f_{gh}\subseteq\Bases$ implies $f_g\subseteq\Bases$ or $f_h\subseteq\Bases$ as desired since $\Bases$ must be a union of orbits.
\end{proof}

For any subset $S\subseteq G$ of indices we introduce the notation \[f_S=\bigcup_{g\in S}f_g\] for a union of orbits. In the proofs below we may sometimes only consider cases where $S\subseteq G-\{e\}$, but since $f_e=\emptyset$ this will never be an issue when quantifying over all possible nonempty unions. Each subset $S$ determines in this way a candidate set of bases of a matroid corresponding to a 2-dimensional tropical subrepresentation. A nonempty union $f_S$ will satisfy the matroid basis exchange axioms if and only if the corresponding 2-vector  
\[\bfn_S=\sum_{g\in S} \bfn_g\] satisfies the tropical Pl\"ucker relations.

\subsection{Example: $\Z_n$} \label{ex:Zn}

 The case $G=\Z_n$ is studied in part in \cite[Section~4]{GM2020}. Here the action on subsets of size $2$ becomes
\[g\cdot\{a,b\}=\{g+a,g+b\},\ g\in \Z_n,\ \{a,b\}\in {{\Z_n}\choose{2}}.\]
Note that the difference $b-a$ is preserved under the $\Z_n$-action, and the distinct orbits are precisely the sets 
 \[f_i=\{\{a\,\, \text{mod}\,\, n,a+i\,\, \text{mod}\,\, n\}|a \in \Z_n\}\]
consisting of subsets of $\Z_n$ of size 2 with a fixed difference $i$ modulo $n$ between the two elements. It is clear from the definition that $f_i=f_j$ whenever $i$ and $j$ are congruent modulo $n$ and $f_i=f_{-i}$ for all $i$. With $n$ odd, the distinct $Z_n$-orbits are  
 \[\O_{\Z_n}= \left\{f_1,\ldots,f_{\frac{n-1}{2}}\right\}\] 
since $n-\frac{n+1}{2}=\frac{n-1}{2}$. With $n$ even, the distinct orbits are 
 \[\O_{\Z_n}= \left\{f_1,\ldots,f_{\frac{n}{2}}\right\}.\] 

Denote by $\Z_n^\times\subseteq \Z_n$ the multiplicative group of units, and let $p$ be a prime. In \cite[Theorem~4.1.3]{GM2020}, Giansiracusa \& Manaker prove that the only subrepresentation of $\BZp$ in dimension two is uniform when $2$ is a primitive root mod $p$ or $p\cong7$ mod $8$ and $2$ has order $\frac{p-1}{2}$ mod $p$. Their argument appeals to studying when $2$ generates the multiplicative quotient group $\Z_n^\times/\langle p-1\rangle$, and applying a weaker version of Proposition~\ref{prop:orbitgh} to build a chain of inclusions for the set of bases of a matroid. They conjecture \cite[Conjecture~4.1.6]{GM2020} that $p$ is prime if and only if the only subrepresentation of $\BZp$ in dimension two corresponds to the uniform rank 2 matroid on $[n]$. We prove this conjecture as a special case of Theorem~\ref{thm:main} in Section~\ref{sec:2dlast}, but as a brief example we can obtain one direction by showing that every two-dimensional tropical subrepresentation of $\BZn$ contains $f_{\Z_n^\times}$ in its set of bases. 

\begin{proposition}\label{lem:fni}     Let $\Bases$ be the set of bases of a matroid $M_\bv=([n],\Bases)$ corresponding to a $\Z_n$-fixed tropical Pl\"ucker vector $\bv\in\bigwedge^2\BZn$. Let $m$ be a positive integer. 
\begin{enumerate}
\item $f_{mi}\subseteq\Bases\implies f_i\subseteq\Bases$.
\item Let $d\mid n$ and $k\in\Z_n$ with $k\neq 0$. Then $f_{kd}\subseteq\Bases\implies f_{d\Z_n^\times}\subseteq\Bases$.
\end{enumerate}
\end{proposition}
\begin{proof}
    We prove (1) by induction. It is a tautology in the case $m=1$. Assume $f_{(m-1)i}\subseteq\Bases\implies f_i\subseteq\Bases$. Then by Proposition~\ref{prop:orbitgh}, \[f_{mi} = f_{(m-1)i+i}\subseteq\Bases\implies f_{(m-1)i}\subseteq\Bases\text{ or }f_i\subseteq\Bases.\] The latter case is immediate. In the former case, we know $f_i\subseteq\Bases$ by the inductive hypothesis.

    For (2), we have by (1) that $f_{kd}\subseteq\Bases\implies f_d\subseteq\Bases$. Then for any $u\in \Z_n^\times$ we must have \[f_d=f_{u^{-1}ud}\subseteq\Bases\implies f_{ud}\subseteq\Bases.\] Thus $f_{d\Z_n^\times}\subseteq\Bases$
\end{proof}

In the setting of the above proposition, since $\Bases$ must be a non-empty union of orbits, we know that for some $k\in \Z_n$, $f_k\subseteq\Bases$. Taking $d=1$ in (2) gives $f_{\Z_n^\times}\subseteq\Bases$. We get the following Corollary.

\begin{corollary}\label{cor:un}
    Every two-dimensional tropical subrepresentation of $\BZn$ contains $f_{\Z_n^\times}$ in its corresponding set of bases. When $n$ is prime then $f_{\Z_n^\times}=f_{[n]}={{[n]}\choose 2}$ and the only 2-dimensional tropical subrepresentation corresponds to the uniform rank 2 matroid $U_{2,n}$ on the ground set $[n]$.
\end{corollary}

\subsection{Classification in dimension 2} \label{sec:2dlast}
Our first theorem classifies all two-dimensional tropical subrepresentations of $\BG$.

\begin{theorem}\label{thm:main}
The two-dimensional tropical subrepresentations of $\BG$ correspond to the Pl\"{u}cker vectors $\bfn_{{G-H}}$ for $H$ a proper subgroup of $G$.
\end{theorem}
\begin{proof}
We begin by showing that, for any proper subgroup $H$ of $G$, the union of orbits $f_{G-H}$ forms the bases of a matroid. Since $H$ is proper, we know $f_{G-H}$ is a non-empty subset of ${{G}\choose{2}}$. 
 Let $\{a_1,a_2\}$ and $\{b_1,b_2\}$ in $f_{G-H}$. Without loss of generality, it is enough to show $\{a_1,b_1\}$ or $\{a_1,b_2\}$ is in $f_{G-H}$.
Assume on the contrary $\{a_1,b_1\}$ and $\{a_1,b_2\}\not\in f_{G-H}$. Notice that since $H$ is a subgroup, both $H$ and $G-H$ are closed under taking inverses so $f_{G-H}\cap f_{H} = \emptyset$ and $f_{G-H}\cup f_{H} = {{G}\choose{2}}$, giving a disjoint union. Thus $\{a_1,b_1\}\in f_{H}$ and $\{a_1,b_2\}\in f_{H}$. Then $a_1^{-1}b_1,a_1^{-1}b_2\in H$, so 
\[(a_1^{-1}b_1)^{-1}a_1^{-1}b_2=b_1^{-1}a_1a_1^{-1}b_2=b_1^{-1}b_2\in H.\]
Thus $\{b_1,b_2\}\in f_H$ contradicting our assumption that $\{b_1,b_2\} \in f_{G-H}$.

For the converse, we will now show that orbit unions of the form $f_{G-H}$ with $H$ a proper subgroup of $G$ are the only $G$-fixed matroidal subsets of ${{G}\choose{2}}$. Let $S\subseteq G-\{e\}$ with $S$ nonempty and assume the union of orbits $f_S=\bigcup_{i\in S} f_i$ is the set of bases indicated by a $G$ fixed Pl\"{u}cker vector, and thus matroidal. Our goal is to show the complement $S^C = G - S$ is a proper subgroup of $G$. 

In the case $S=G-\{e\}$ we have $S^C=\{e\}$, which is a proper subgroup as desired. In this case, $f_S={{G}\choose{2}}$ and thus $\bfn_S$ corresponds to the uniform matroid. 

Assume now $S$ is a proper nonempty subset of $G-\{e\}$. Let $a,b\in S^C$. Then $f_a,f_b\subseteq f_{S^C}$. We will use the subgroup test to show that $S^C$ is a subgroup of $G$.  Assume on the contrary that $ab^{-1}\not \in S^C$.  Then $ab^{-1} \in S$ giving $f_{ab^{-1}}\subseteq f_S$. Thus $f_a$ or $f_{b^{-1}}=f_b\subseteq f_S$ by Proposition~\ref{prop:orbitgh}. So $a$ or $b\in S$ which contradicts our assumption $a,b\in S^C$. Thus $S^C$ is a subgroup of $G$.    
\end{proof}

\begin{remark}
    The uniform matroid corresponds to choosing $H=\{e\}$. 
 \end{remark}

In the cyclic case $G=\Z_n$, setting $n$ to be a prime gives $f_{\Z_n^\times} = {{[n]}\choose{2}}$ and confirms Giansiracusa \& Manaker's Conjecture.

\begin{corollary}\cite[Conjecture~4.1.6]{GM2020}
For $\BZn$, $n$ is prime if and only if the only two-dimensional subrepresentation corresponds to the uniform matroid. 
\end{corollary}

\subsection{Examples}
\begin{example}[$\B{[Q_8]}$]

    We use the usual presentation \[Q_8=\langle -1, i, j, k | (-1)^2=1, i^2=j^2=k^2=ijk=-1\rangle\] of the quaternions. There are 4 orbits of the action of $Q_8$ on subsets of $Q_8$ of size 2 by left multiplication:
    \begin{align*}
      f_{-1}=&\{\{1,-1\},\{i,-i\},\{j,-j\},\{k,-k\}\},\\
      f_{i}=&\{\{1,i\},\{-1,-i\},\{i,-1\},\{-i,1\},\{j,-k\},\{-j,k\},\{k,j\},\{-k,-j\}\},\\
      f_{j}=&\{\{1,j\},\{-1,-j\},\{i,k\},\{-i,-k\},\{j,-1\},\{-j,1\},\{k,-i\},\{-k,i\}\},\\
      f_{k}=&\{\{1,k\},\{-1,-k\},\{i,-j\},\{-i,j\},\{j,i\},\{-j,-i\},\{k,-1\},\{-k,1\}\}.
    \end{align*}
    By Proposition~\ref{prop:orbitgh}, $f_{-1}\subseteq\Bases\implies f_i\subseteq\Bases$ since $i^2=-1$. Combining this with similar statements for $j$ and $k$, we obtain $f_{-1}\subseteq\Bases\implies f_i\cup f_j\cup f_k\subseteq\Bases$.

    We can write $i=jk$, so by Proposition~\ref{prop:orbitgh} we know that $f_i\subseteq\Bases\implies f_j\text{ or }f_k\subseteq\Bases$.
    Similarly for $f_j$ and $f_k$, we obtain \[f_j\subseteq\Bases\implies f_i\text{ or }f_k\subseteq\Bases,\] and \[f_k\subseteq\Bases\implies f_i\text{ or }f_j\subseteq\Bases.\]
    Combining all of these statements together, it can be checked that the only matroidal sums of orbits are the following:
    \begin{align*}
      f_{Q_8-\langle k\rangle}=&f_i\cup f_j\\
      f_{Q_8-\langle j\rangle}=&f_i\cup f_k\\
      f_{Q_8-\langle i\rangle}=&f_j\cup f_k\\
      f_{Q_8-\langle -1\rangle}=&f_i\cup f_j\cup f_k\\
      f_{Q_8-\langle 1\rangle}=&f_{-1}\cup f_i\cup f_j\cup f_k
    \end{align*}
    These correspond to the subgroups of $Q_8$.
    
\end{example}

\begin{example}[$\B{[D_n]}$]
    We use the presentation $D_n=\langle \rho, \sigma | \rho^n=\sigma^2=e$, $\sigma\rho\sigma=\rho^{-1}\rangle$. The orbits of the left action of $D_n$ on ${D_n}\choose{2}$ are as follows:
    \begin{enumerate}
        \item $f_{\rho^i}=\{\{e,\rho^i\},\{\rho,\rho^{i+1}\},\cdots,\{\rho^{n-1},\rho^{i-1}\},\{\sigma,\sigma\rho^i\},\{\sigma\rho,\sigma\rho^{i+1}\},\cdots,\{\sigma\rho^{n-1},\sigma\rho^{i-1}\}\}$ 
        for $0<i<\frac{n}{2}$,
        \item $f_{\rho^{\frac{n}{2}}}=\{\{e,\rho^{\frac{n}{2}}\},\{\rho,\rho^{\frac{n}{2}+1}\},\cdots,\{\rho^{\frac{n}{2}-1},\rho^{n-1}\},\{\sigma,\sigma\rho^{\frac{n}{2}}\},\{\sigma\rho,\sigma\rho^{\frac{n}{2}+1}\},\cdots,\{\sigma\rho^{\frac{n}{2}-1},\sigma\rho^{n-1}\}\}
    $ if $n$ is even,
        \item $f_{\sigma\rho^i}=\{\{e,\sigma\rho^i\},\{\rho,\sigma\rho^{i-1}\},\cdots,\{\rho^{n-1},\sigma\rho^{i+1}\}\}$, for $0\le i<n$.
    \end{enumerate}
    If $n$ is odd, there are $\frac{n-1}{2}$ of the first type of orbit, each of which has order $2n$, while there are $n$ of the third type of orbit, each of which has order $n$. This gives a total of $2n(\frac{n-1}{2})+n(n)=n(2n-1)={{2n}\choose{2}}$ elements. If $n$ is even, there are $\frac{n}{2}-1$ of the first type of orbit, each of which has order $2n$, one of the second type of orbit of order $n$, and $n$ of the third type of orbit, which has order $n$. This gives us a total of $2n(\frac{n}{2}-1)+1(n)+n(n)=n(2n-1)={{2n}\choose{2}}$ elements.

    By Theorem~\ref{thm:main}, matroidal unions of orbits correspond to subgroups of $D_n$. Subgroups of $D_n$ are classified as follows (see for instance \cite[Theorem 3.1]{Conrad} or \cite{Cavior}):
    \begin{enumerate}
        \item[(i)] $\langle \rho^d\rangle$, $d\mid n$,
        \item[(ii)] $\langle \rho^d,\sigma\rho^i\rangle$, $d\mid n$ $0\le i<d$.
    \end{enumerate}
    
Since the subgroups depend on the factorization of $n$, we restrict our example to the case $D_p$ where $p$ is prime for simplicity. The corresponding unions of orbits are:

    $d=1$: $\langle\rho\rangle$ corresponds to $f_{D_p-\langle\rho\rangle}=f_{\sigma}\cup f_{\sigma\rho}\cup \cdots\cup f_{\sigma\rho^{n-1}}$;

    $d=p$: $\langle e\rangle$ corresponds to the uniform matroid, $f_{D_p-\langle e\rangle} = f_{D_p}$, and

    \hspace{36pt}$\langle\sigma\rho^i\rangle$, $0\le i<p$, corresponds to $f_{D_p-\langle\sigma\rho^i\rangle}=f_{D_p}-f_{\sigma\rho^i}$.
    
    For a specific example of the composite case, we turn to $D_4$:

    $d=1$: $\langle\rho\rangle$ corresponds to $f_{D_4-\langle\rho\rangle}=f_{\sigma}\cup f_{\sigma\rho}\cup f_{\sigma\rho^2}\cup f_{\sigma\rho^3}$;

    $d=2$: $\langle\rho^2\rangle$ corresponds to $f_{D_4-\langle\rho^2\rangle}=f_{\rho}\cup f_{\sigma}\cup f_{\sigma\rho}\cup f_{\sigma\rho^2}\cup f_{\sigma\rho^3}$;
    
    \hspace{36pt}$\langle\rho^2,\sigma\rangle$ corresponds to $f_{D_4-\langle\rho^2,\sigma\rangle}=f_{\rho}\cup f_{\sigma\rho}\cup f_{\sigma\rho^3}$;

    \hspace{36pt}$\langle\rho^2,\sigma\rho\rangle$ corresponds to $f_{D_4-\langle\rho^2,\sigma\rho\rangle}=f_{\rho}\cup f_{\sigma}\cup f_{\sigma\rho^2}$;

    $d=4$: $\langle e\rangle$ corresponds to $f_{D_4-\langle e\rangle}=f_{\rho}\cup f_{\rho^2}\cup f_{\sigma}\cup f_{\sigma\rho}\cup f_{\sigma\rho^2}\cup f_{\sigma\rho^3}$ (uniform);

    \hspace{36pt} $\langle\sigma\rangle$ corresponds to $f_{D_4-\langle\sigma\rangle}=f_{\rho}\cup f_{\rho^2}\cup f_{\sigma\rho}\cup f_{\sigma\rho^2}\cup f_{\sigma\rho^3}$;

    \hspace{36pt} $\langle\sigma\rho\rangle$ corresponds to $f_{D_4-\langle\sigma\rho\rangle}=f_{\rho}\cup f_{\rho^2}\cup f_{\sigma}\cup f_{\sigma\rho^2}\cup f_{\sigma\rho^3}$;

    \hspace{36pt} $\langle\sigma\rho^2\rangle$ corresponds to $f_{D_4-\langle\sigma\rho^2\rangle}=f_{\rho}\cup f_{\rho^2}\cup f_{\sigma}\cup f_{\sigma\rho}\cup f_{\sigma\rho^3}$;

    \hspace{36pt} $\langle\sigma\rho^3\rangle$ corresponds to $f_{D_4-\langle\sigma\rho^3\rangle}=f_{\rho}\cup f_{\rho^2}\cup f_{\sigma}\cup f_{\sigma\rho}\cup f_{\sigma\rho^2}$.

\end{example}

\section{3-dimensional Subrepresentations} \label{sec:3d}

The problem in dimension 3 offers more combinatorial complexity. We present in this section some properties of orbits and computations of three dimensional subrepresntations of $\BG$. We identify tropical subrepresentations of $\BG$ corresponding to subgroups of index larger than 2. In the cyclic case we find additional tropical subrepresentations of $\BZn$ that do not fit into this correspondence. We also identify a specific collection of orbits, indexed by $\Z_n^\times$, that are contained in the set of bases of any matroid corresponding to a subrepresentation of $\BZn$.

\subsection{Orbits and their properties} \label{sec:3d1} As in dimension 2, the $G$ action on $\BG$ extends to the basis of $\bigwedge^3\BG$ as
\[g\cdot \be_a\wedge \be_b \wedge \be_c = \be_{ga}\wedge \be_{gb}\wedge \be_{gc},\ g\in G, \{a,b,c\}\in {{G}\choose{3}}\]
and extends linearly to an action on $\bigwedge^3\BG$. The equivalent $G$ action on ${{G}\choose{3}}$ is similar. Following the approach in dimension 2, $G$-orbits in this setting are given by sets of the form
\[f_{g,h} = \{\{a,ag,ah\} | a\in G\} = \{\{a,b,c\} | a^{-1}b=g\text{ and }a^{-1}c = h\}\]
for any $g,h\neq e$ with $g\neq h$. It is obvious from the definition that $f_{g,h} = f_{h,g}$. We write $\bfn_{g,h}\in\bigwedge^3\BG$ for the correspinding  Pl\"ucker vector. We will investigate similar properties of the orbits that can be deduced in the 3-dimensional setting. The first is analogous to Proposition~\ref{prop:orbit-1}.

\begin{proposition}\label{prop:dim3orbit}
    For all $g,h\in G$, $f_{g,h}=f_{g^{-1},g^{-1}h}=f_{h^{-1},h^{-1}g}$.
\end{proposition} 
\begin{proof}
    To show $f_{g,h}\subseteq f_{g^{-1},g^{-1}h}$, let $\{a,b,c\}\in f_{g,h}$ be given. Then $g=a^{-1}b$ and $h=a^{-1}c$. Notice $g^{-1} = b^{-1}a$ and
    \begin{align*}
        g^{-1}h &=(b^{-1}a)(a^{-1}c)\\
        &=b^{-1}c
    \end{align*}
    thus 
    $\{a,b,c\} = \{b,a,c\}\in f_{g^{-1},g^{-1}h}$. The other containment is similar. Finally, a similar proof gives the third equality.
\end{proof}

Let $A=\{a_1,a_2,a_3\}$ and $B=\{b_1,b_2,b_3\}$ in $\Bases$ be bases of a matroid $(G,\Bases)$. There are now four cases when analyzing basis exchange to replace $a_3$:

\begin{enumerate}
    \item[Case 1:] $A=B$. In this case $A-B$ is empty and basis exchange succeeds.
    \item[Case 2:] $A,B$ share two elements, none of which are $a_3$. Without loss of generality, let $a_1=b_1$ and $a_2=b_2$. Then $x=a_3$, $y=b_3$, and this forces $A-a_3+b_3$ to be in $\Bases$. However, this set is $\{a_1,a_2,b_3\}$, but since $a_1=b_1$ and $a_2=b_2$ the set is just $\{b_1,b_2,b_3\}=B$.
    \item[Case 3:] $A,B$ share a single element that isn't $a_3$. Without loss of generality, let $a_1=b_1$. Then replacing $a_3$ gives $\{a_1,a_2,b_2\}$ or $\{a_1,a_2,b_3\}\in\Bases$.
    \item[Case 4:] $A,B$ are disjoint. Then replacing $a_3$ gives$\{a_1,a_2,b_1\}$, $\{a_1,a_2,b_2\}$ or $\{a_1,a_2,b_3\}\in\Bases$.
\end{enumerate}
Note that only cases 3 and 4 result in the inclusion of new bases. The following proposition shows how basis exchange in this setting interacts with the orbits of the $G$-action.

\begin{proposition}\label{prop:dim3exchange}
     Let $\Bases$ be the set of bases of a matroid $M_\bv=(G,\Bases)$ corresponding to a $G$-fixed tropical Pl\"ucker vector $\bv\in\bigwedge^3\BG$. Then $f_{g,h}\cup f_{g',h'}\subseteq\Bases\implies f_{g,g'}\text{ or } f_{g,h'}\subseteq\Bases$.
 \end{proposition}
 \begin{proof}
     First consider the case where $g,h,g',$ and $h'$ are distinct. We know $\{0,g,h\}\in f_{g,h}$ and $\{0,g',h'\}\in f_{g',h'}$. Basis exchange gives $\{0,g,g'\}$ or $\{0,g,h'\}\in\Bases$ and $f_{g,g'}$ or $f_{g,h'}\subseteq\Bases$. In the case where $g,h,g',h'$ are not distinct, there are 6 cases:
     \begin{itemize}
         \item the cases $g=h$ and $g'=h'$ aren't possible based on how the orbits are defined,
         \item $g=g'\implies f_{g',h'}=f_{g,h'}$,
         \item $g=h'\implies f_{g',h'}=f_{g,g'}$,
         \item $h=g'\implies f_{g,h}=f_{g,g'}$,
         \item $h=h'\implies f_{g,h}=f_{g,h'}$.
     \end{itemize}
\end{proof}

As an immediate corollary, we have a property of orbits that is analogous to the product property of Proposition~\ref{prop:orbitgh}.

 \begin{corollary}\label{prop:dim3sumdiff} Let $\Bases$ be the set of bases of a matroid $M_\bv=(G,\Bases)$ corresponding to a $G$-fixed tropical Pl\"ucker vector $\bv\in\bigwedge^3\BG$. Then $f_{g,gh}\subseteq\Bases\implies f_{g,g^{-1}}\text{ or } f_{g,h}\subseteq\Bases$.
\end{corollary}
 \begin{proof}
     By Proposition~\ref{prop:dim3orbit}, we know $f_{g,gh}=f_{g^{-1},h}$. The result follows directly from applying Proposition~\ref{prop:dim3exchange}.
 \end{proof}

A straightforward application this basis exchange property provides an important reduction of exponents in the indices of the orbits. 

\begin{proposition}\label{prop:gg2} Let $\Bases$ be the set of bases of a matroid $M_\bv=(G,\Bases)$ corresponding to a $G$-fixed tropical Pl\"ucker vector $\bv\in\bigwedge^3\BG$. Then $f_{g,g^k}\subseteq\Bases\implies f_{g,g^2}\subseteq\Bases$ for $k\geq 2$.
    
\end{proposition}
\begin{proof}
    We proceed by induction. The case $k=2$ is a tautology. Assume for induction hypothesis that $f_{g,g^{k-1}}\subseteq\Bases\implies f_{g,g^2}\subseteq\Bases$. Then \[f_{g,g^{k}}\subseteq\Bases\implies f_{g,g^{-1}}\text{ or } f_{g,g^{k-1}}\subseteq\Bases\] by Corollary~\ref{prop:dim3sumdiff}. In the first case, we know $f_{g,g^{-1}}=f_{g^{-1},g}$ and \[f_{g^{-1},g}=f_{g,g^2}\] using Proposition~\ref{prop:dim3orbit}. The second case follows immediately from the induction hypothesis.
\end{proof}

\subsection{Subrepresentations in dimension 3}
 Recall that we defined a notation $f_S=\bigcup_{g\in S}f_g$ for the union of orbits in dimension 2 indexed by a subset $S\subseteq G$. We adopt a similar notation in dimension 3:
    \[f_S=\bigcup_{g,h,g^{-1}h\in S} f_{g,h}.\]
Our main theorem in dimension 3 gives an attempt to extend our dimension 2 classification to three dimensional tropical subrepresentations of $\BG$ coming from proper subgroups of $G$.

\begin{theorem}\label{thm:dim3subgroups}
    Let $G$ be a finite group, and let $H$ be a subgroup of $G$ with $[G:H]>2$. Then there is a matroid $M_\bv=(G,f_{G-H})$ corresponding to a $G$-fixed tropical Pl\"ucker vector $\bv\in\bigwedge^3\BG$.
\end{theorem}

\begin{proof}
    We prove this by contradiction. Assume on the contrary for some $X=\{x_1,x_2,x_3\},Y=\{y_1,y_2,y_3\}\in\Bases$ we can apply basis exchange to force the inclusion of an orbit containing a difference in $H$. Note that by definition, $x_1^{-1}x_2,\ x_1^{-1}x_3,\ x_2^{-1}x_3,\ y_1^{-1}y_2,\ y_1^{-1}y_3,\ y_2^{-1}y_3$ must be in $G-H$. Their inverses must also be in $G-H$, since it's closed under taking inverses.

    Note that neither $X$ nor $Y$ contains a difference in $H$ by our assumption. The case where $[G:H]=2$ is impossible here. If $H$ has index 2, the there are two cosets: $H$ and $aH$. Assume on the contrary that $x_1^{-1}x_2,\ x_1^{-1}x_3,\ x_2^{-1}x_3$ in $G-H=aH$. Then $x_1^{-1}x_2=ah_1$ and $x_1^{-1}x_3=ah_2$ with $h_1,h_2\in H$. Then, $x_2^{-1}x_3=x_2^{-1}x_1x_1^{-1}x_3=(x_1^{-1}x_2)^{-1}x_1^{-1}x_3=(ah_1)^{-1}ah_2=h_1^{-1}h_2\in H$.
    
    Now, we assume $[G:H]>2$.
    \begin{enumerate}
        \item $X$ and $Y$ share one element, $x_1=y_1$. $\{x_1,x_2,x_3\},\{y_1,y_2,y_3\}\in\Bases\implies \{x_1,x_2,y_2\}$ or $\{x_1,x_2,y_3\}\in\Bases$.
        Since we've assumed we can force the inclusion of an orbit containing a difference in $H$, there must be a difference in $H$ in both sets. So the first set must include a difference in $H$. There are three possibilities:
        \begin{enumerate}
            \item $x_1^{-1}x_2\in H$ is a contradiction;
            \item $x_1^{-1}y_2\in H$ is also a contradiction since $x_1^{-1}y_2=y_1^{-1}y_2\in H$
            \item $x_2^{-1}y_2\in H$ is the only remaining case.
        \end{enumerate}
        Similarly we can show that $x_2^{-1}y_3\in H$. However, then $y_2^{-1}y_3=y_2^{-1}x_2x_2^{-1}y_3=(x_2^{-1}y_2)^{-1}x_2^{-1}y_3\in H$       giving a contradiction.
        \item $X$ and $Y$ are disjoint. $\{x_1,x_2,x_3\},\{y_1,y_2,y_3\}\in\Bases\implies \{x_1,x_2,y_1\}$ or $\{x_1,x_2,y_2\}$ or $\{x_1,x_2,y_3\}\in\Bases$.
        Since $d$ must divide a difference in each set, it must divide a difference in the first set; there are again three possibilities:
        \begin{enumerate}
            \item $x_1^{-1}x_2\in H$ is a contradiction;
            \item if $x_1^{-1}y_1\in H$, note that $x_1^{-1}y_2$ or $x_1^{-1}y_3\in H$ would give a contradiction as in case (c) above, so it must be the case that $x_2^{-1}y_2$ and $x_2^{-1}y_3\in H$, which gives us the contradiction we were trying to avoid; and finally
            \item the case where $x_2^{-1}y_1\in H$ gives a similar contradiction.
        \end{enumerate}
    \end{enumerate}
\end{proof}

Unlike in Theorem~\ref{thm:main}, this is not an equivalence. Indeed there exist matroids in dimension three that do not come from complements of subgroups, suggesting the theory is more involved. In Section~\ref{sec:3dcyclic} we will study the cyclic case to identify some of these subrepresentations.

\subsection{Results in the cyclic case} \label{sec:3dcyclic} Our final theorems in dimension 3 are restricted to the cyclic case $G=\Z_n$. We classify a number of tropical subrepresentations that appear and give a collection of orbits indexed by the generators of the group that will be always be subsets of the bases of a matroid corresponding to a tropical subrepresentations. While not a complete classification, these results make apparent the combinatorics that appear in higher dimension and hopefully provide a path for further results in higher dimensions. 

As in the dimension 2 example given in Section~\ref{ex:Zn}, we will work using additive notation. The orbits take the form 
\[f_{i,j} = \{\{a,a+i,a+j\} | a\in \Z_n\} = \{\{a,b,c\} | b-a=i\text{ and }c-a = j\}\]
for $0<i<j<n$. We begin with two lemmas to aid in the proof of Theorem~\ref{thm:3d}. This first lemma provides a useful bound on the indices of $f_{i,j}$.

\begin{lemma}\label{lem:smallestindex}
    Let $G=\Z_n$ be a cyclic group and $u\in\Z_n^\times$ a unit. Let $f_{i,j}$ be an orbit of the $G$ action on ${{G}\choose{3}}$, with $0<i<j<n$. Let $I=\{k\mid f_{ku,lu}=f_{i,j}, 0<k<l<n\}$. Then $I$ has an element $s$ with $3s\le n$.
\end{lemma}

\begin{proof}

    Recall $f_{i,j}=f_{-i,j-i}=f_{-j,i-j}$ by Proposition~\ref{prop:dim3orbit}. Note also that $f_{-i,j-i} = f_{j-i,-i}$ since we can switch the order of indices.

    Write $i=au,\ j-i=bu,\ -j=cu$. We know $0<a,b<n$ since $0<i<j<n$. We also know $-n<c<0$ since $0<j<n,$ so we let $c'=c+n$ to ensure $0<c'<j$. Then, $a,b,c'\in I$.

    We also know that $i+(j-i)-i=0$, so $a+b+c=0$. Then $a+b+c'=a+b+c+n=n$. 
    
    Let $s$ be the smallest element of $\{a,b,c'\}\subseteq I$. Note that $0<s$. Then, $0<3s\le a+b+c'=n$.
\end{proof}

In the proof of Theorem~\ref{thm:3d} we will need to split into even and odd cases for $n$. The next lemma will help us deal with the even case.

\begin{lemma}\label{lem:dim3n/2}
    Let $\Bases$ be the set of bases of a matroid $M_\bv=([n],\Bases)$ corresponding to a $\Z_n$-fixed tropical Pl\"ucker vector $\bv\in\bigwedge^3\BZn$. If $n=2k$, $k>2$, and $u\in\Z_n^\times$, then  \[f_{u,ku}\subseteq\Bases\iff f_{u,(k+1)u}\subseteq\Bases.\] 
\end{lemma}
\begin{proof}
    Assuming $f_{u,ku}\subseteq\Bases$, we know by Proposition~\ref{prop:dim3orbit} that
    \[f_{u,ku}=f_{-u,(k-1)u}=f_{ku,(k+1)u}\]
    and
    \[f_{u,(k+1)u}=f_{-u,ku}=f_{(k-1)u,ku}\] 
    (using the assumption that $n=2k$ and working modulo $n$ in the indices).
    When $k>2$, $u\ne ku\ne -u\ne (k-1)u$, and we can use basis exchange in the following way: 
    \[f_{u,ku}\cup f_{-u,(k-1)u}\subseteq\Bases\implies f_{-u,ku}\text{ or }f_{(k-1)u,ku}\subseteq\Bases\] 
    Notice the orbits on the left are $f_{u,ku}$ and the ones on the right are $f_{u,(k+1)u}$, so we know 
    \[f_{u,ku}\subseteq\Bases\implies f_{u,(k+1)u}\subseteq\Bases\]
    The other direction is similar.
\end{proof}

\begin{theorem}\label{thm:3d}
    Let $G=\Z_n$ be a cyclic group, and let $\Bases$ be the set of bases of a matroid $M_\bv=([n],\Bases)$ corresponding to a $\Z_n$-fixed tropical Pl\"ucker vector $\bv\in\bigwedge^3\BZn$. Then, for any $u\in\Z_n^\times$, $f_{u,2u}\subseteq\Bases$.
\end{theorem}

\begin{proof}
    
    Let $I=\{k\mid f_{ku,lu}\subseteq\Bases\text{ for }0<k<l<n\}$.
    Since $I$ must be non-empty as the set of bases is nonempty and contains at least one orbit, we know that $I$ must have a least element $s$. We will show that $s=1$. Assume on the contrary $s>1$.
    
    We know $3s\le n$ by Lemma~\ref{lem:smallestindex}. Since $1<s$ we have $2s+1<3s\le n$. Since $s\in I$ there is an orbit  $f_{su,tu}\subseteq\Bases$ for some $t>s>1$, so we can simply long divide $t=sq+r$ for some $1\le q$, $0\le r<s$. 
    
    In the case $r=0$, we have $t=sq$ and thus $f_{su,(sq)u}\subseteq\Bases$ for some $2\le q$.  
    Proposition~\ref{prop:gg2} gives $f_{su,(2s)u}\subseteq\Bases$, so we must have basis elements
    \[\{0,su,(2s)u\},\{u,(s+1)u,(2s+1)u\}\in f_{su,(2s)u}\subseteq\Bases.\] Since $2s+1<n$ these are disjoint. Basis exchange gives 
    \[\{0,u,(s+1)u\}\in\Bases\text{ or }\{u,su,(s+1)u\}\in\Bases\text{ or }\{u,(s+1)u,(2s)u\}\in\Bases.\]
    If $\{0,u,(s+1)u\}\in \Bases$ then $f_{u,(s+1)u}\subseteq \Bases$ and $1\in I$. If $\{u,su,(s+1)u\}\in \Bases$ then $f_{(s-1)u,su}\subseteq\Bases.$ By Proposition~\ref{prop:dim3orbit}, $f_{(s-1)u,su}=f_{(1-s)u,u}=f_{u,(1-s)u}\subseteq\Bases$ and again $1\in I$. Finally, if $\{u,(s+1)u,(2s)u\}\subseteq \Bases$, then $f_{su,(2s-1)u}\subseteq\Bases$. By Proposition~\ref{prop:dim3orbit}, \[f_{su,(2s-1)u}=f_{-su,(s-1)u}=f_{(s-1)u,-su}\subseteq\Bases\] and $s-1\in I$. In all three cases of basis exchange, we obtain a contradiction since $s$ is not the least element of $I$.

    In the case $r\ne 0$, we know $f_{su,(sq+r)u}\subseteq\Bases$. Thus by Corollary~\ref{prop:dim3sumdiff} we must have $f_{su,(-sq)u}$ or $f_{su,ru}\subseteq\Bases$. If $f_{su,(-sq)u}$, we can again use Proposition~\ref{prop:gg2} to show $f_{su,(2s)u}\subseteq\Bases$ and the logic in the previous case to arrive at a contradiction. Otherwise, we know $r<s\in I$, again giving a contradiction.

    So it must be the case that $s=1$, and $f_{u,tu}\subseteq\Bases$ for some $t>1$. Thus $f_{u,2u}\subseteq\Bases$ by Proposition~\ref{prop:gg2}. 
    
\end{proof}

In particular, When $n=p$ is a prime this theorem tells us that all orbits of the form $f_{a,a^2}$ for $a=1,\ldots,p-1$ are contained in the set of bases for every matroid corresponding to a fixed tropical Pl\"ucker vector $\bv\in\bigwedge^3\BZn$. However, not all orbits are of this form. Our next two result provides the full set of bases in several cases for tropical subrepresentations of $\BZn$. 

\begin{theorem}\label{thm:3d2}
For a tropical Pl\"ucker vector $\bv\in\bigwedge^3\BZn$, write $\Bases_\bv$ for the set of bases of its corresponding matroid. 
    \begin{enumerate}
        \item For any $u\in\Z_n^\times$, if $n$ is odd, there exists a tropical Pl\"ucker vector $\bv\in\bigwedge^3\BZn$ with $\Bases_\bv= {{[n]}\choose{3}}-f_{u,ku}$ if and only if $k\ne -1,2,\frac{n+1}{2}$; 
        \item For any $u\in\Z_n^\times$, if $n$ is even, there exists a tropical Pl\"ucker vector $\bv\in\bigwedge^3\BZn$ with $\Bases_\bv={{[n]}\choose{3}}-f_{u,ku}$ if and only if $k\ne -1,2, \frac{n}{2},\frac{n}{2}+1$.
    \end{enumerate}
\end{theorem}

\begin{proof}
    To begin, assume on the contrary $\Bases = {{[n]}\choose{3}}-f_{u,ku}$ form the bases of a matroid for at least one of the cases $k=-1,k=2$, or $k=\frac{n}{2},k=\frac{n}{2}+1$ for $n$ even, or  $k=\frac{n+1}{2},|G|$ for $n$ odd. First, $k=-1,k=2$ give an immediate contradiction since $f_{u,-u}=f_{u,2u}\subseteq\Bases$ by Proposition~\ref{prop:gg2}. When $n$ is even, either of $k=\frac{n}{2},k=\frac{n}{2}+1$ give a contradiction since $f_{u,(\frac{n}{2})u}\subseteq\Bases\iff f_{u,(\frac{n}{2}+1)u}\subseteq\Bases$ by Lemma~\ref{lem:dim3n/2}. Finally when $n$ is odd, if $k=\frac{n+1}{2}$, then $2(\frac{n+1}{2})u$ is congruent to $u$ modulo $n$. Since $\frac{n+1}{2}\in\Z_n^\times$, we know by Proposition~\ref{prop:gg2} that $f_{u,(\frac{n+1}{2})u}=f_{(\frac{n+1}{2})u,u}=f_{(\frac{n+1}{2})u,2(\frac{n+1}{2})u}\subseteq\Bases$.
    
    In the opposite direction, we need to show ${{[n]}\choose{3}}-f_{u,ku}$ is a matroid provided $k\ne-1,k\ne2; k\ne\frac{n}{2}$ and $k\ne\frac{n}{2}+1$ for $n$ even; and $k\ne\frac{n+1}{2}$ for $n$ odd. We assume on the contrary that ${{[n]}\choose{3}}-f_{u,ku}$ is not a matroid, and show that $k$ has to be one of the above values. Since ${{[n]}\choose{3}}-f_{u,ku}$ is nonempty, by the strong basis exchange axiom ${{[n]}\choose{3}}-f_{u,ku}$ is not a matroid if and only if for some $X=\{x_1,x_2,x_3\}$, some $Y=\{y_1,y_2,y_3\}\in{{[n]}\choose{3}}-f_{u,ku},$ and some $i\in X-Y,$ we have that for all $j\in Y-X$, $X-i+j\in f_{u,ku}$.

    If $X$ and $Y$ share more than one element, no new bases must be included (see the discussion in Section~\ref{sec:3d1}). Thus two cases remain:
    \begin{enumerate}
        \item $X$ and $Y$ share one element; without loss of generality set $x_1=y_1$. In this case, basis exchange is not satisfied if and only if $\{x_1,x_2,y_2\},\{x_1,x_2,y_3\}\in f_{u,ku}$.
        
        \item $X$ and $Y$ are disjoint. In this case, basis exchange is not satisfied if and only if $\{x_1,x_2,y_1\},\{x_1,x_2,y_2\},\{x_1,x_2,y_3\}\in f_{u,ku}$.

    \end{enumerate}

    In the first case, we know \[\{x_1,x_2,y_2\},\{x_1,x_2,y_3\}\in f_{u,ku}=f_{-u,(k-1)u}=f_{-ku,(1-k)u}\] by Proposition~\ref{prop:dim3orbit}. Notice that the difference $x_2-x_1$ is in at least two of the following 3 sets: $\{u,ku\},\{-u,(k-1)u\},
    \{-ku,(1-k)u\}$. Since $u\in\Z_n^\times$, we can divide out by $u$ and write the 12 possibilities in the following way:
    \begin{itemize}
        \item $1=-k$ and $k=-1$ give $k=-1$.
        \item $1=-k+1$ and $-1=k-1$ give $k=0$, which isn't possible.
        \item $1=k$ and $-1=-k$ give $k=1$, which isn't possible.
        \item $1=k-1$ and $-1=-k+1$ give $k=2$.
        \item $k=-k$ gives $2k=0$; since $k\ne 0$ this case only occurs when $n$ is even and $k=\frac{n}{2}$.
        \item $k=-k+1$ and $k-1=-k$ give $2k=1$, which only occurs when $n$ is odd and $k=\frac{n+1}{2}$.
        \item $k-1=-k+1$ gives $2k=2$; since $k\ne 1$ this case only occurs when $n$ is even and $k=\frac{n}{2}+1$.
    \end{itemize}

    Thus these are the only possible values of $k$:
    
    \begin{itemize}
        \item $k=-1$
        \item $k=2$
        \item $k=\frac{n}{2}$, $|G|$ even
        \item $k=\frac{n+1}{2}$, $|G|$ odd
        \item $k=\frac{n}{2}+1$, $|G|$ even
    \end{itemize}

    In the second case case, we know \[\{x_1,x_2,y_1\},\{x_1,x_2,y_2\},\{x_1,x_2,y_3\}\in f_{u,ku}=f_{-u,(k-1)u}=f_{-ku,(1-k)u}\] by Proposition~\ref{prop:dim3orbit}. Notice that the difference $x_2-x_1\in\{u,ku\},\{-u,(k-1)u\},\text{ and }
    \{(1-k)u,-ku\}$. This is a special case of the first case, so no new values of $k$ arise.

\end{proof}

\begin{remark}
    This guarantees non-uniform matroids but does not proclude there from being other types of non-uniform matroids. Based on SAGE computations, the first example of a prime with matroids excluding more than one orbit is $\B{[\Z_{13}]}$.
\end{remark}

\subsection{Examples}
We will examine some groups of small order using our results to find all dimension 3 subrepresentations. 

\begin{example}[$\B{[\Z_6]}$]

    This example was included in \cite{GM2020}. We arrive at the same result using our methods. There are four orbits:
    \begin{align*}
        f_{1,2}=f_{1,5}=f_{4,5}=&\{\{0,1,2\},\{1,2,3\},\{2,3,4\},\{3,4,5\},\{0,4,5\},\{0,1,5\}\}\\
        f_{1,3}=f_{2,5}=f_{3,4}=&\{\{0,1,3\},\{1,2,4\},\{2,3,5\},\{0,3,4\},\{1,4,5\},\{0,2,5\}\}\\
        f_{1,4}=f_{2,3}=f_{3,5}=&\{\{0,1,4\},\{1,2,5\},\{0,2,3\},\{1,3,4\},\{2,4,5\},\{0,3,5\}\}\\
        f_{2,4}=&\{\{0,2,4\},\{1,3,5\}\}.
    \end{align*}
    We know by Theorem~\ref{thm:3d} that $f_{1,2}=f_{4,5}\subseteq\Bases$ for every matroid. By Proposition~\ref{prop:dim3exchange}, $f_{1,2}\subseteq\Bases\implies f_{1,4}$ or $f_{2,4}\subseteq\Bases$. Furthermore, by Lemma~\ref{lem:dim3n/2}, $f_{1,3}\subseteq\Bases\iff f_{1,4}\subseteq\Bases$. It can be checked that the following are the only matroidal unions of orbits.
    \begin{align*}
    f_{1,2}&\cup f_{2,4}\\
    f_{1,2}&\cup f_{1,3}\cup f_{1,4}\\
    f_{1,2}&\cup f_{1,3}\cup f_{1,4}\cup f_{2,4}
    \end{align*}    
\end{example}

Note that the first matroid corresponds to the subgroup with index $3$ in $\Z_6$ as in Theorem~\ref{thm:dim3subgroups}.

\begin{example}[$\B{[\Z_7]}$]

    The orbits in this example are:
    \[f_{1,2}=f_{1,6}=f_{5,6}=\{\{0,1,2\},\{1,2,3\},\{2,3,4\},\{3,4,5\},\{4,5,6\},\{0,5,6\},\{0,1,6\}\}\]
    \[f_{1,3}=f_{2,6}=f_{4,5}=\{\{0,1,3\},\{1,2,4\},\{2,3,5\},\{3,4,6\},\{0,4,5\},\{1,5,6\},\{0,2,6\}\}\]
    \[f_{1,4}=f_{3,4}=f_{3,6}=\{\{0,1,4\},\{1,2,5\},\{2,3,6\},\{0,3,4\},\{1,4,5\},\{2,5,6\},\{0,3,6\}\}\]
    \[f_{1,5}=f_{2,3}=f_{4,6}=\{\{0,1,5\},\{1,2,6\},\{0,2,3\},\{1,3,4\},\{2,4,5\},\{3,5,6\},\{0,4,6\}\}\]
    \[f_{2,4}=f_{2,5}=f_{3,5}=\{\{0,2,4\},\{1,3,5\},\{2,4,6\},\{0,3,5\},\{1,4,6\},\{0,2,5\},\{1,3,6\}\}.\]
    By Theorem~\ref{thm:3d}, $f_{1,2},\ f_{2,4}=f_{3,5}$, and $f_{3,6}=f_{1,4}\subseteq\Bases$. Again, basis exchange can be used to show that $f_{1,2}\text{ and } f_{2,4}\subseteq\Bases\implies f_{1,3}$ or $f_{1,5}\subseteq\Bases$. It can be checked that the following are the only matroidal unions of orbits.
    \begin{align*}
      f_{1,2}&\cup f_{1,3}\cup f_{1,4}\cup f_{2,4}\\
      f_{1,2}&\cup f_{1,4}\cup f_{1,5}\cup f_{2,4}\\
      f_{1,2}&\cup f_{1,3}\cup f_{1,4}\cup f_{1,5}\cup f_{2,4}
    \end{align*}
\end{example}
    
\begin{example}[$\B{[S_3]}$]
We use the presentation $S_3=\langle \rho, \sigma | \rho^3=\sigma^2=e$, $\sigma\rho\sigma=\rho^{-1}\rangle$. The orbits of the left action of $S_3$ on ${S_3}\choose{3}$ are as follows.
    \begin{align*}
        f_{\rho,\rho^2}=&\{\{e,\rho,\rho^2\},\{\sigma,\sigma\rho,\sigma\rho^2\}\}\\
        f_{\rho,\sigma}=f_{\rho^2,\sigma\rho}=f_{\sigma,\sigma\rho}=&\{\{e,\rho,\sigma\},\{\rho,\rho^2,\sigma\rho^2\},\{e,\rho^2,\sigma\rho\},\{e,\sigma,\sigma\rho\},\{\rho^2,\sigma\rho,\sigma\rho^2\},\{\rho,\sigma,\sigma\rho^2\}\}\\f_{\rho,\sigma\rho}=f_{\rho^2,\sigma\rho^2}=f_{\sigma\rho,\sigma\rho^2}=&\{\{e,\rho,\sigma\rho\},\{\rho,\rho^2,\sigma\},\{e,\rho^2,\sigma\rho^2\},\{\rho,\sigma,\sigma\rho\},\{e,\sigma\rho,\sigma\rho^2\},\{\rho^2,\sigma,\sigma\rho^2\}\}\\f_{\rho,\sigma\rho^2}=f_{\rho^2,\sigma}=f_{\sigma,\sigma\rho^2}=&\{\{e,\rho,\sigma\rho^2\},\{\rho,\rho^2,\sigma\rho\},\{e,\rho^2,\sigma\},\{\rho^2,\sigma,\sigma\rho\},\{\rho,\sigma\rho,\sigma\rho^2\},\{e,\sigma,\sigma\rho^2\}\}
    \end{align*}

Based on SAGE computations, there are 5 matroidal sums of orbits.
    \begin{align*}
      f_{\rho,\rho^2}&\cup f_{\rho,\sigma}=f_{S_3-<\sigma\rho^2>}\\
      f_{\rho,\rho^2}&\cup f_{\rho,\sigma\rho}=f_{S_3-<\sigma>}\\
      f_{\rho,\rho^2}&\cup f_{\rho,\sigma\rho^2}=f_{S_3-<\sigma\rho>}\\
      f_{\rho,\sigma}&\cup f_{\rho,\sigma\rho}\cup f_{\rho,\sigma\rho^2}\\
      f_{\rho,\rho^2}&\cup f_{\rho,\sigma}\cup f_{\rho,\sigma\rho}\cup f_{\rho,\sigma\rho^2}\text{ (uniform)}
    \end{align*}

\end{example}
\bibliographystyle{amsalpha}
\bibliography{matroid}
\nocite{*}

\end{document}